\documentclass[a4paper,12pt]{article}
\usepackage{epsf,epsfig}
\usepackage{graphics,color}
\usepackage{amsmath}
\usepackage{amssymb}
\usepackage{cite}
\usepackage{verbatim}
\usepackage{graphicx}
\usepackage{amsthm}
\usepackage{textcomp}
\usepackage{subfig}

\newcommand{\io}{\int_\Omega}
\newcommand{\duav}[1]{\langle{#1}\rangle}

\newcommand{\R}{\mathbb{R}}
\newcommand{\RR}{\mathbb{R}}
\newcommand{\HI}{\mathcal{H}}
\newcommand{\Om}{\Omega}

\def\supp{\,{\rm supp \ }}

\def\argmin{\,{\rm argmin \ }}

\def\H{\mathcal H}

\def\eps{\varepsilon}

\numberwithin{equation}{section}
\mathchardef\emptyset="001F

\newtheorem{theorem}{Theorem}[section]
\newtheorem{prop}[theorem]{Proposition}

\newtheorem{lemma}[theorem]{Lemma}

\theoremstyle{definition}
\newtheorem{remark}[theorem]{Remark}

\newcommand{\betti}{\color{black}}  
\newcommand{\bl}{\color{black}}

\numberwithin{equation}{section}
\begin{document}

\title{ A rigorous sharp interface limit \\ of a diffuse interface model related to tumor growth}

\author{
Elisabetta Rocca\\
{\sl Dipartimento di Matematica, Universit\`a degli Studi di Pavia}\\ 
{Via Ferrata 5,  I-27100 Pavia, Italy}\\
 {\rm E-mail:~~\tt  elisabetta.rocca@unipv.it}
\and 
Riccardo Scala\\
{\sl Dipartimento di Matematica, Universit\`a degli Studi di Pavia}\\ 
{Via Ferrata 5,  I-27100 Pavia, Italy}\\
 {\rm E-mail:~~\tt  riccardo.scala@unipv.it}
}


\maketitle
\begin{abstract}
 In this paper we study the  rigorous \bl sharp interface limit of a diffuse interface model  related \bl to the dynamics of 
 tumor growth, when a parameter $\eps$,  representing  \bl the interface thickness  between the tumorous and non tumorous cells\bl, 
 tends to zero. More in particular, we analyze  here a
 gradient-flow type model arising from a modification of the recently introduced model for tumor growth dynamics in \cite{HZO} (cf. also \cite{Hil}). 
 Exploiting the techniques related to both gradient-flows and gamma convergence, we recover a condition on the interface $\Gamma$ 
 relating the chemical and double-well potentials, the mean curvature, and the normal velocity. 
\end{abstract}


\vspace{2mm}

\noindent {\bf AMS (MOS) subject clas\-si\-fi\-ca\-tion:}  49J40, 82B24, 35K46, 35K57, 35R11, 92B05. \bl


\section{Introduction}
\label{sec:intro}
The morphological evolution of a growing solid tumor is the result of the dynamics of a complex system that includes many nonlinearly interacting factors, including cell-cell and cell-matrix adhesion, mechanical stress, cell motility and angiogenesis, just to name a few. Numerous mathematical models have been developed to study various aspects of tumor progression and this has been an area of intense research interest (see the recent reviews by Fasano et al. \cite{FBG}, Graziano and Preziosi \cite{GP07}, Friedman et al. \cite{FBM07}, Bellomo et al. \cite{BLM08}, Cristini et al. \cite{CEtAl08}, and Lowengrub et al. \cite{LEtAl10}). The existing models can be divided into two main categories: continuum models and discrete models. We concentrate on the former ones. There the necessity of dealing with multiple interacting constituents has led to the consideration of diffuse-interface models based on continuum mixture theory (see, for instance, \cite{CLW09} and references therein). 
In the diffuse approach, sharp interfaces are replaced by narrow transition layers that arise due to differential adhesive
forces among the cell-species. The main advantages of the diffuse interface formulation are: 
\begin{itemize}
\item[-] it eliminates the need to enforce complicated boundary conditions across the tumor/host tissue and other species/species interfaces that would have to be satisfied if the interfaces were assumed sharp, and
\item[-] it eliminates the need to explicitly track the position of interfaces, as is required in the sharp interface framework.
\end{itemize}

Such models generally consist of Cahn-Hilliard equations with transport and reaction terms which govern various types of cell concentrations. Here we consider only one tumor specie and we denote the tumorous phase  by \bl $u$. The reaction terms depend on the nutrient concentration (e.g., oxygen), denoted here by $\sigma$, which obeys an advection-reaction-diffusion equation coupled with the Cahn-Hilliard equations. The cell velocities satisfy a generalized Darcy's (or Brinkman's) law where, besides the pressure gradient, also appears the so-called Korteweg force due to cell concentration. 

While there exist quite a number of numerical simulations of diffuse-interface models of tumor growth (cf., e. g., \cite[Chap. 8]{CL10}, \cite{CLW09, HZO, WLFC08}), there are still only a few contributions to the mathematical analysis of the models. The first contributions in this direction dealt with the case where the nutrient is neglected, which then leads to the so-called Cahn-Hilliard-Hele-Shaw system (see, e.g., \cite{LTZ13, JWZ15}).  Moreover, we refer to the paper \cite{GetAl15} where a new model for tumor growth including different densities is introduced and a formal sharp interface limit is performed. Finally, in the recent contribution \cite{FGR} the model introduced in \cite{HZO} (where the velocity is neglected) was rigorously analyzed concerning well-posedness, regularity, and asymptotic behavior. We also refer to the recent papers \cite{CGRS1, CGRS2}, in which various viscous approximations of the state system have been studied analytically,  and \cite{CGRS3} where a first optimal control problem for tumor growth models has been investigated. \bl  Hence, the existing literature is just at a preliminary step towards the theoretical analysis of more refined models. 

 Regarding the transition from diffuse to sharp interfaces, \bl several results already regard some formal passages to the sharp interface limit (cf., e.g., \cite{GetAl15, Hil}), but, up to our knowledge, no rigorous theorems are proved for such coupled systems. Only very recently in \cite{DFRSS} we investigated the existence of weak solutions and some rigorous sharp interface limit (in a simplified case) for a model introduced in \cite{CWSL14} where both velocities (satisfying a Darcy law with Korteweg term) and multispecies tumor fractions as well as the nutrient evolutions are taken into account. However these are very partial results, because only the coupling between the \bl Cahn-Hilliard equation and Darcy  law for the velocities  is considered \bl and, for example, the physically meaningful case of a double-well potential in the Cahn-Hilliard equation (cf.~\eqref{def:v}) \bl cannot be accounted for in \cite{DFRSS}. 

Hence, the main goal of  this paper is to \bl  perform a rigorous sharp interface limit as the thickness of the interface goes to zero, in the spirit of what is already known for the  standard \bl Cahn-Hilliard equation (cf., e.g., \cite{Le}  and references therein). The model under consideration is a variant of the diffuse interface tumor growth model introduced in \cite{HZO} where we first write down the system as a {\em gradient-flow system} and then use refined results of gamma convergence already exploited  in \cite{SS}, and applied to the Cahn-Hilliard equation in \cite{Le}.
 Another possibility would be the one of considering the known results for Cahn-Hilliard equations by \cite{C96}, trying to extend them to the coupled Cahn-Hilliard-Darcy system (first neglecting the nutrient) in the spirit of \cite{AL14}, and then trying to get possibly weaker results for the complete system. This is a work in progress. 
 
Here, more \bl in particular, we  aim to let \bl $\eps$ tend to zero in the following PDE system in $\Omega\times (0,T)$, where $\Omega\subset \RR^3$ denotes a {\em regular} domain,
\begin{align}\label{PDE}
 \begin{cases}
  &\dot u+A^{ s\bl }v=R(u,v,\sigma)\\
  &\dot \sigma+A^s(\sigma)=-R(u,v,\sigma),
 \end{cases}
\end{align}
coupled with suitable initial conditions, where we choose $R(u,v,\sigma)= 2\sigma +u-v$ and 
\begin{equation}
\label{def:v}
v=-\eps\Delta u+\frac{W'(u)}{\eps}, \quad W(u)=(u^2-1)^2\,. 
\end{equation}
\bl 
Here $A$ denotes the Laplace operator with Neumann homogeneous boundary conditions and $A^s$ stands for \bl its power $s$ with  $s\geq 1$\bl. 
Then, introducing the auxiliary variable $\varphi:=u+\sigma$, we rewrite  \eqref{PDE} as the 
following gradient-flow system:
\begin{align}
 (\dot\varphi,\dot\sigma)=-\nabla_{X^\eps\times Y^\eps}E^\eps(\varphi,\sigma),
\end{align}
where $X^\eps\times Y^\eps=H^{-s}_n(\Om)\times L^2(\Om)$,  being $H^{-s}_n(\Om)$ the dual space of $H^s_n(\Omega)=D(A^{s/2})$, \bl and where the energy functional $E^\eps$ is defined as 
\begin{align*}
 E^\eps(\varphi, \sigma):=\begin{cases}
                           M^\eps(\varphi-\sigma)+F(\varphi, \sigma)&\text{if }(\varphi,\sigma)\in L^2(\Om)\times H_n^s(\Om)\betti\text{ and }\varphi-\sigma\in H^1(\Om)\\
                           +\infty&\text{otherwise.}
                          \end{cases}
\end{align*}
Here $M^\eps$ is the  functional defined on $H^1(\Om)$ by
\begin{align}
 M^\eps(u):=\int_\Om\left(\frac{W(u)}{\eps}+\frac{\eps}{2}|\nabla u|^2\right)\, dx,
\end{align}
and $F$ is the function on $L^2(\Om)\times H_n^s(\Om)$ given by
\begin{align}
 F(\varphi, \sigma):=\int_\Om\left(\sigma\varphi+\frac{|\sigma|^2}{2}\right)dx+\frac{a_s(\sigma,\sigma)}{2}
\end{align}
\bl 
being $a_s$ the bilinear form associated to $A^s$ defined as
\begin{align}
 a_s(u,v):=\int_{\Om}A^{s/2}u\;A^{s/2}v dx.
\end{align}

Let us notice that the operator $A^s$ can be interpreted as a {\em nonlocal} contribution to our energy functional modeling nonlocal interactions between cells (cf. \cite{WLFC08, GL1, GL2} for a physical interpretation). 

The presence of such regularization,  entailing that  $v^\eps\in L^2(0,T;H^s_n(\Omega))$ for every $\eps>0$ in \eqref{def:v}, enables us to avoid, in case $s>3/2$, to make further assumptions on the convergence of the approximated chemical potential $v^\eps$ related to the so-called {\sl equipartition of energy} (cf.~Hyp.~(HP1bis) in Section~\ref{hypotheses}), which are instead needed in case $s\in [1,3/2]$. \bl 
Therefore the result in \cite[Theorem 3.2]{RoTon}, in case  $s>3/2$, or the Hyp.~(HP1bis) related to the {\sl equipartition of energy} in case $s\in [1,3/2]$, \bl ensures that the limiting function $v$ satisfies
$$v=-c_Wk\;\;\;\;\;\;\text{ on }\Gamma,$$
where $c_W=\int_{-1}^1 W(s)\, ds$ and $k$ is the mean curvature of the limiting interface $\Gamma$ between the two open sets $\Omega^+$ and $\Omega^-$ where $u$ takes values $u\equiv 1$ and $u\equiv -1$ (the pure phases), respectively.\bl
We also emphasize that in order to get such result, we need to assume some regularity of the limit interface. In particular, the limit interface must be at least of class $C^3$ in the time-space in order that the derivation of its motion law can be obtained.

Let us notice that the choice we make of the coupling function $R$ is almost obliged from the fact that we aim to write down the system as a gradient-flow. Possibly, more general functions  should  be taken into account in order to accomplish with the tumor growth model introduced in \cite{HZO} (cf.~also \cite{FGR, CGRS1, CGRS2}), but we cannot treat these cases with our techniques here. \bl 

In order to obtain our results we proceed as follows:
\begin{itemize}
\item[Step 1] We prove the well-posedness of the system \eqref{PDE} for $\eps>0$ by means of a passage to the limit in a suitable time-discrete approximation scheme.
\item[Step 2] We consider the functional 
\begin{align*}
 E^0(\varphi,\sigma):=\begin{cases}
                    M^0(\varphi-\sigma)+F(\varphi,\sigma)&\text{if }(\varphi-\sigma)\in BV(\Om,\{-1,1\}), \, \sigma\in H_n^s(\Om)\\
                       +\infty&\text{otherwise}\\
                      \end{cases}
\end{align*}
where $M^0$ is defined as
\begin{align}
 M^0(u):=\begin{cases}
                       c_W \H^{n-1}(\Gamma)&\text{if }u\in BV(\Om,\{-1,1\})\\
                       +\infty&\text{otherwise,}\\
         \end{cases}
\end{align}
 and we demonstrate that the functionals $E^\eps$  gamma converge to $E^0$ with respect to the $L^1(\Omega)$-topology when $\eps$ tends to $0$.
 \item[Step 3] We state the regularity assumptions we need (in particular on the interface $\Gamma$ between the two phases $u\equiv -1$ and $u\equiv 1$) in order to prove our main result mainly stating that the limit functions (in proper functional spaces) $\varphi$, $\sigma$ and $v$ of $\varphi^\eps$, $\sigma^\eps$ and $v^\eps$  satisfying \eqref{PDE} are solutions of the following system on some time interval $[0,T^*]$:
   \begin{align}\nonumber
  &2\dot\Gamma(t)=-A^s v(t)+\varphi(t)+\sigma(t)-v(t)\hbox{ on }\Omega,\\
  \nonumber
  &\dot\sigma(t)=-A^s\sigma(t)+v(t)-\varphi(t)-\sigma(t)\hbox{ on }\Omega,\\
  \nonumber
&v(t)=-c_Wk(t) \text{ on }\Gamma,\quad A^s v(t)=\varphi(t)+\sigma(t)-v(t)\hbox{ on $\Om^+\cup\Om^-$},
\end{align}
where $\dot\Gamma$ denotes here the normal velocity of the interface $\Gamma$.
\item[Step 4] In the special cases when $s=1$ ($A^s=-\Delta$) and $s=2$ ($A^s=\Delta^2$), then, we can also deduce that 
$$\left[\frac{\partial v}{\partial n}\right](t)=-2\dot\Gamma(t)\;\;\;\;\mathcal H^2-\text{a.e. on }\Gamma \hbox{ and for a.e. $t\in[0,T^*]$}$$
and 
$$\left[\frac{\partial \Delta v}{\partial n}\right](t)=-2\dot\Gamma(t)\;\;\;\;\mathcal H^2-\text{a.e. on }\Gamma \hbox{ and  for a.e. $t\in[0,T^*],$}$$
respectively, where $[\cdot]$ denotes the jump of the functions across $\Gamma$. 
\end{itemize}
\betti 
 Let us emphasized that the techniques of proof are quite elementary and strongly based on previous results on the $\Gamma$ convergence of the Modica-Mortola functional \cite{MoMo} and on the convergence of the solutions to the Cahn-Hilliard equation to the Mullins-Sekerka flow \cite{Le}.
  \bl
 
Finally, let us conclude by mentioning that, although molecular mechanisms and cell-scale migration dynamics are well described, the variable empirical and qualitative observations of tumor invasion and response to therapy illustrate the critical need for biologically realistic and predictive multiscale mathematical models that integrate tumor proliferation and invasion with microvascular effects and microenvironmental substrate gradients. Such complex systems, dominated by large numbers of processes and highly nonlinear dynamics, are difficult to approach by experimental methods alone and can typically be better understood with appropriate mathematical models and sophisticated computer simulations, in addition and complementary to experimental investigations.
By focusing on these common elements, mathematical modeling aims to contribute to the prevention, diagnosis and treatment of this complex disease. The ultimate goal is for modeling and simulation to aid in the development of individualized therapy protocols to minimize patient suffering while maximizing treatment effectiveness. 

More in particular, in larger scale systems, diffuse interface continuum methods provide a good modeling approach and then it is clear that the study of the corresponding sharp interface limits would be an important validation of the models.  In this direction the present contribution is a first step toward the validation of previous works where only formal asymptotic limits were performed (cf., e.g., \cite{GetAl15} and \cite{Hil}). Moreover, we believe that the same techniques could be applied in the future to different type of complex system dynamics like Liquid Crystals' evolution for example. \bl 

\noindent
{\bf Plan of the paper.} 
In Section \ref{setting} we fix some notation and preliminaries. In Section \ref{energies} we introduce our energies functionals and prove the preliminary results about  gamma convergence. Then we prove the well-posedness of our diffuse interface model \eqref{PDE}  for a fixed $\eps>0$ \bl in Section \ref{appr_gradient}. The last part, Section~\ref{sharp}, is devoted to study the limit  of equations \eqref{PDE} as $\eps$ vanishes. Such section is divided in two parts, in the first one we fix some conventions and hypotheses on our setting, in the second one we prove our main result.

\section{Space setting and notation}\label{setting}
Let $\Om$ be a smooth and bounded open subset of $\R^3$.  If $X$ stands for a Hilbert space, we denote by $(\cdot, \cdot)_X$ the scalar product in $X$, while we denote by $\duav{\cdot,\cdot}$ the duality pairing between every two dual spaces. \bl 

 \paragraph{Powers of positive operators.} We denote by $V$ the Hilbert space $H^1(\Om)$ and by $H=L^2(\Om)$, the latter endowed with scalar product $(\cdot,\cdot)$ and norm $\|\cdot\|$. 
 Then, for any $\zeta\in V'$, set 
\begin{align}\label{defiwo}
  &\zeta_\Omega:=\frac{1}{|\Omega|}\duav{\zeta,1},\\
 \label{defiV}
  &{\cal V}':=\{\zeta\in V': \zeta_\Omega=0\},
   \qquad {\cal V}:=V\cap {\cal V}'.
\end{align}
The above notation $\cal{V}'$ is just suggested for the
sake of convenience; indeed, we mainly see ${\cal V}$, ${\cal V}'$ as 
(closed) subspaces of $V$, $V'$, inheriting their norms, 
rather than as a couple of spaces in duality. We introduce the realization of the Laplace operator
with homo\-gene\-ous Neumann boundary con\-diti\-ons as 
\begin{equation}\label{defiB2}
  A:V\to V', \qquad
   \duav {A u,v}:=\io \nabla u\cdot\nabla v\,dx
   \hbox{ for~~$u,v\in V$.}
\end{equation}
Clearly, $A$ maps $V$ onto ${\cal V}'$ 
and its restriction to ${\cal V}$ is an isomorphism of ${\cal V}$ 
onto ${\cal V}'$. Let us denote by ${\cal N}:{\cal V}'\to {\cal V}$ 
the inverse of $A$, so that, for any $u\in V$ and
$\zeta\in {\cal V}'$, there holds
\begin{equation}\label{relaN}
  \duav{A u,{\cal N}\zeta}
  =\duav{A{\cal N}\zeta,u}
    =\duav{\zeta,u}.
\end{equation}
By using the Poincar\'e-Wirtinger inequality, we can easily
show that the norm
\begin{equation}\label{coercNM}
  \Big(\int_{\Omega}\vert\nabla\big({\cal N} \zeta\big)\vert^2\Big)^{1/2}
   =\duav{\zeta,{\cal N} \zeta}^{1/2}\hbox{ for $\zeta\in {\cal V}'$}
\end{equation}
is equivalent to the norm $\Vert \zeta\Vert_{V'}$ and we
will use this norm, when it is convenient.

Define $H^2_n(\Om):=\{w\in H^2(\Omega):\quad\partial_n w=0\quad\hbox{on }\partial\Omega\}$, where $\partial_n$ is the derivative with respect to~the outward normal to $\partial\Omega$, and introduce also the following spaces ${\cal W}:=H_n^2\cap {\cal V}'$ and ${\cal H}:=H\cap {\cal V}'$. Now, it is also possible to restrict the operator ${\cal N}$ to a new isomorphic operator (always called ${\cal N}$) from ${\cal H}$ to ${\cal W}$: it maps $v\in {\cal H}$ into the unique function ${\cal N} v\in \cal W$ such that
$$-\Delta({\cal N } v )\,=\,v \quad\hbox{a.e. in }\Omega,\quad\hbox{and}\quad \partial_n({\cal N} v)\,=\,0 \quad\hbox {a.e. on }\partial\Omega,\,\quad\io{\cal N} v\,=\,0.$$
Note that any solution $\Phi$ to
\begin{equation}
\label{nlaplace}
-\Delta\Phi=v\quad\hbox{a.e. in }\,\Omega\quad\hbox{and }\,\,\partial_n\Phi=0\quad\hbox{a.e. on }\,\partial\Omega,
\end{equation}
corresponding to a $v\in{\cal H}$, can be written as $\Phi={\cal N}v+m$, where $m$ is the mean-value of $\Phi$.

We consider then every positive power $A^s$ of the positive operator $A$, with $s>0$ that \bl can be also defined as follows: if $e_i\in L^2(\Om)$ is a basis of eigenfunctions with eigenvalues $\lambda_i$, $i\in \mathbb N$, then it holds, for all $u\in D(A^s)$,
\begin{equation}\label{s_laplace}
u=\sum_ic_ie_i\;\;\;\Rightarrow \;\;\;A^su=\sum_i\lambda_i^sc_ie_i. 
\end{equation}
%
%

Moreover, for $1<s<2$ we denote by 
\begin{equation}\label{Hs}
H^s_n(\Om):=D(A^{s/2}) 
\end{equation}
and by ${\cal H}_n^s(\Omega):=\{u\in H^s_n(\Omega)\,:\, \int_\Om u=0\}$. We consider on ${\cal H}_n^s(\Omega)$ the scalar product 
\begin{align}\label{as}
 a_s(u,v):=( A^{s/2}u,A^{s/2}v)\quad \forall u,\,v\in \mathcal H^s_n(\Om)\,.
\end{align}
In the space $H^s_n(\Om)$ we can also consider the equivalent norm
\begin{align}
 \label{as+l2}
 a_s(u,u)^{1/2}+\|u\|\quad \forall u\in H^s_n(\Om)\,.
\end{align}

We define 
\begin{align*}
 H^{-s}_n(\Om)=\{f\in (\mathcal H^s_n(\Om))':\exists \,g\in {\cal H}^s_n(\Om)\text{ such that }\langle f,\varphi\rangle=a_s( g,\varphi)\;\;\;\forall \varphi\in {\cal H}^s_n(\Om)\}.
\end{align*}
It can be observed that for all $f\in H^{-s}_n(\Om)$, it holds $g=A^{-s}f$.
We endow the space $H_n^{-s}(\Om)$ with the scalar product
$$( f_1,f_2)_{H_n^{-s}}:=( A^{-s/2}f_1,A^{-s/2}f_2).$$ 
With such product $H_n^{-s}(\Om)$ is a Hilbert space with norm denoted by $\|\cdot\|_{H_n^{-s}}$. 
\bl 

Let $U\subset\subset \Om$ and $K=\bar U$ be compact. Denote by $H_K^s(\Om)$ the space of functions $u\in H^s(\Om)$ such that $\supp u\subset K$.
The interpolation space of order $s\in(0,1)$ between $H$ and $W:=H^2_n(\Om)$ is $[H,W]_s$. From the inclusions $H^2_0(U)\subset H^2_K(\Om)\subset W\subset H^2(\Om)$ and thanks to the fact that $H^s_n(\Om)=D(A^{s/2})=[H,W]_s$ it is possible to prove that
\begin{align}
 H_K^s(\Om)\subset H_n^s(\Om)\subset H^s(\Om).
\end{align}
This follows from the facts that $[H,H^2_0(U)]_s=H_K^s(\Om)$, $s\in(1/2,1)$ (see \cite[Section 4.5]{Tay}). In particular it is seen that smooth functions with compact support in $\Om$ belong to  $H_n^s(\Om)$ for all $s\in(0,1)$. 

\bl
\paragraph{Properties of operators defined on $\Gamma$.} \bl Let $\Gamma$ be a smooth interface between the two open sets $\Om^+$ and $\Om^-$. 
Let us consider the map $T:H^{1/2}(\Gamma)\rightarrow H^1(\Om)$ such that $T(f)=\tilde f$, where, for $f\in H^{1/2}(\Gamma)$, $\tilde f\in \betti \mathcal V'$ \bl  is defined as the null-mean value solution of the problem
\begin{align}\label{f:tilde}
\Delta \tilde f=0\;\;\text{ on }\Om^+\cup\Om^-,\;\;\;\tilde f=f\;\;\text{ on }\Gamma,\;\;\;\partial_n\tilde f=0\;\;\text{ on }\partial\Om. 
\end{align}
%
We consider the inner product on $H^{1/2}(\Gamma)$
\begin{align}\label{scalar.H12}
( u,v)_{H^{1/2}(\Gamma)}=\int_\Om \nabla T(u)\cdot \nabla T(v)dx\;\;\;\forall u,v \in H^{1/2}(\Gamma),
\end{align}
which induces the seminorm $\|\cdot\|_{H_n^{1/2}(\Gamma)}$. It is easy to observe that $\|f\|_{H_n^{1/2}(\Gamma)}=0$ if and only if $f$ is constant, and thus if we note by $\sim$ the equivalence relation $f_1\sim f_2$ iff $f_1-f_2$ is constant on $\Gamma$, we see that $H^{1/2}(\Gamma)/\sim$ is a Hilbert space with scalar product \eqref{scalar.H12}. Hence  we have \bl 
$$H_n^{1/2}(\Gamma)=H^{1/2}(\Gamma)/\sim.$$

 We  denote \bl  by $H_n^{-1/2}(\Gamma)$ the dual space of $H_n^{1/2}(\Gamma)$.
We now introduce the Laplace operator restricted to $\Gamma$, namely $-\Delta_\Gamma:H_n^{1/2}(\Gamma)\rightarrow H_n^{-1/2}(\Gamma)$, defined as
\begin{align}\label{Delta:G}
 -\Delta_\Gamma(f):=\left[\frac{\partial \tilde f}{\partial n}\right]\;\;\;\;\forall f \in H_n^{1/2}(\Gamma).
\end{align}
Here we have use the following notation: for all $f\in H^{1/2}(\Gamma)$, if $f^\pm$ are the two restrictions of $\tilde f$ in \eqref{f:tilde} to $\Om^\pm$ respectively, $[\frac{\partial \tilde f}{\partial n}]\in H_n^{-1/2}(\Gamma)$ is the jump of the normal derivative of $\tilde f$ on $\Gamma$, i.e. $[\frac{\partial \tilde f}{\partial n}]:=\frac{\partial f^+}{\partial n}-\frac{\partial f^-}{\partial n}$. This is well defined in $H_n^{-1/2}(\Gamma)$ and coincides with the distribution
\begin{align}\label{2.7}
 \langle\left[\frac{\partial \tilde f}{\partial n}\right],\varphi\rangle=\int_{\Om}\nabla \tilde f\cdot\nabla \bar\varphi dx\;\;\;\;\forall \varphi \in H_n^{-1/2}(\Gamma),
\end{align}
where $\bar\varphi\in V$ is an arbitrary extension of $\varphi$. 
In particular, we can always choose $\bar \varphi=T(\varphi)$, so that, taking $\varphi=f$ it is also readly seen that  $-\Delta_\Gamma$ is a positive operator.

It is immediately seen that the scalar product \eqref{scalar.H12} can be equivalently rewritten as
\begin{align}\label{scalar.H12*}
 (u,v)_{H_n^{1/2}(\Gamma)}\bl=-\langle (\Delta_\Gamma u),v\rangle\;\;\;\forall u,v \in H^{1/2}(\Gamma).
\end{align}
 The following lemma is proved in \cite[Lemma 2.1]{Le}:
\begin{lemma}\label{lemma1}
 \begin{itemize}
  \item[(i)] For all $u\in H_n^{-1/2}(\Gamma)$ there exists a unique $u^*\in H_n^{1/2}(\Gamma)$ such that $\|u\|_{H_n^{-1/2}(\Gamma)}=\|u^*\|_{H_n^{1/2}(\Gamma)}$ and
  $$\langle u,v\rangle=( u^*,v)_{H_n^{1/2}(\Gamma)},$$
  for all $v\in H_n^{1/2}(\Gamma)$. Moreover $-\Delta_\Gamma u^*=u$, so that by uniqueness we can write $u^*=-\Delta_\Gamma^{-1}u$.
  \item[(ii)] $H_n^{-1/2}(\Gamma)$ is a Hilbert space with inner product
  $$ \langle u,v\rangle=(\Delta_\Gamma^{-1}u,\Delta_\Gamma^{-1}v)_{H_n^{1/2}(\Gamma)},\bl$$
  for all $u,v\in H_n^{-1/2}(\Gamma)$.
 \end{itemize}

\end{lemma}

Notice that $f\in H_n^{-1/2}(\Gamma)$ can be naturally seen as an element $T^*f\in V'$ by the relation
\begin{align}\label{relation}
\langle T^*f,\varphi\rangle=\langle f,\varphi\llcorner_\Gamma\rangle, 
\end{align}
for all $\varphi\in V$.  
Thus $H_n^{-1/2}(\Gamma)$ is isomorphic to a subspace of $V'_\Gamma\subset H_n^{-1}(\Om)$ defined as 
 \begin{equation}
 V'_\Gamma:=\{u\in  H_n^{-1}(\Om)\cap \mathcal D'(\Gamma)\}=\{ u\in  H_n^{-1}(\Om):\text{supp}u\subset \Gamma\}.
 \end{equation}
 The isomorphism is exactly the map $T^*:H_n^{-1/2}(\Gamma)\rightarrow V'_\Gamma$.
 We define $$V_\Gamma:=A^{-1}(V'_\Gamma)=\{u\in H_n^{1}(\Om):\text{supp}Au\subset \Gamma\}.$$
 The space $V_\Gamma$ is isomorphic to $H_n^{1/2}(\Gamma)$ via the isomorphism $T:H_n^{1/2}(\Gamma)\rightarrow V_\Gamma$ introduced in \eqref{f:tilde}.

Assertion (i) of the previous lemma has the following consequence: 
\begin{lemma}\label{lemma:invDelta1}
 Let $f\in H^{-1/2}(\Gamma)$, then $T\circ{(-\Delta_\Gamma)^{-1} f}=\mathcal N\circ T^*f$.
\end{lemma}

\begin{proof}
For all $\varphi\in H^1_n(\Omega)$ we have by (i) of Lemma \ref{lemma1} 
\begin{align}
  \langle T^*f,\varphi\rangle=\langle f,\varphi\llcorner_\Gamma\rangle=\langle-\Delta_\Gamma(-\Delta_\Gamma)^{-1}f,\varphi\rangle,
\end{align}
and, denoting by $g=T\circ{(-\Delta_\Gamma)^{-1} f}$, the last term equals $
 \int_\Gamma[\frac{\partial g}{\partial n}]\varphi d\mathcal H^{d-1}$. Therefore integrating by parts, or in other words using formula \eqref{2.7}, we get
\begin{align}
 \langle T^*f,\varphi\rangle=\int_\Omega\nabla g\cdot\nabla\varphi dx.
\end{align}
Therefore by the definition of $\mathcal N$ and the arbitrariness of $\varphi\in H^1_n(\Omega)$ it follows that $g=\mathcal N T^*f$, that is the thesis. 
\end{proof}
%
 Lemma \ref{lemma:invDelta1} says exactly that $(-\Delta_\Gamma)^{-1}=T^{-1}\circ\mathcal N\circ T^*$, i.e., $-\Delta_\Gamma=(T^*)^{-1}\circ A\circ T$.

\begin{lemma}\label{lemma:invDelta2}
 Let $f\in H_n^{-1}(\Om)\cap H$ and let $g:=\mathcal N f$. Then $[\frac{\partial g}{\partial n}]=0$ on $\Gamma$.
\end{lemma}
\begin{proof}
 By definition we have
 \begin{align}
\langle f,\varphi\rangle=\int_\Om \nabla g\cdot \nabla\varphi dx,  
 \end{align}
 for all $\varphi\in H^1_n(\Omega)$. On the other hand we have
 \begin{align*}
  \langle f,\varphi\rangle&=\int_\Om f\varphi dx=\int_{\Om^+} A\mathcal Nf\varphi dx+\int_{\Om^-} A\mathcal Nf\varphi dx\\
\nonumber
&=-\int_\Gamma\left[\frac{\partial g}{\partial n}\right]\varphi d\mathcal H^{d-1}+\int_{\Om} \nabla g\cdot \nabla\varphi dx.
 \end{align*}
from which the thesis follows.
\end{proof}

\section{Energies and preliminary results}\label{energies}
Let $\Om$ be an open and bounded smooth set in $\R^3$ and $s\geq 1$\bl. We consider the functional  $M^\eps$ defined on $V$ as
\begin{align}
 M^\eps(u):=\int_\Om\eps^{-1}W(u)+\frac{\eps}{2}|\nabla u|^2dx.
\end{align}
We also define, for all  $u\in L^1(\Om)$, the energy
\begin{align}
 M^0(u):=\begin{cases}
                       c_W \H^{n-1}(\Gamma)&\text{if }u\in BV(\Om,\{-1,1\})\\
                       +\infty&\text{otherwise}\\
         \end{cases}
\end{align}
where $\Gamma$ is the boundary of the set $\{u=1\}$, that is, the interface between the two phases of $u$, namely $\{u=\pm 1\}$, and 
$$c_W:=\int_{-1}^1W(s)ds.$$
The following Theorem is well-known and first proved by Modica and Mortola \cite{MoMo}:
\begin{theorem}\label{Modica}
The functionals $M^\eps(u)$
 gamma \bl converge in $L^1(\Om)$ to $M^0(u)$.
\end{theorem}

Let us denote by $F$ the functional on $H\times H^s_n(\Om)$  given by
\begin{align}
 F(\varphi, \sigma):=\langle\sigma,\varphi\rangle+\frac{a_s(\sigma,\sigma)}{2}+\int_\Om\frac{\sigma^2}{2}dx.
\end{align}
Let $\eps>0$, the energy functional $E^\eps$ is defined as
\begin{align}\label{E^eps}
 E^\eps(\varphi, \sigma):=\begin{cases}
                           M^\eps(\varphi-\sigma)+F(\varphi, \sigma)&\text{if }(\varphi,\sigma)\in H\times H^s_n(\Om)\betti\text{ and }\varphi-\sigma\in H^1(\Omega)\\
                           +\infty&\text{otherwise.}\\
                          \end{cases}
\end{align}
Standard estimates show that there exists a constant $C=C(\eps)>0$ such that 
\begin{equation}\label{coerc}
 E^\eps(\varphi,\sigma)\geq C\big(\|\varphi\|^2_{V}+\|\sigma\|^2_{H^s_n}\big),
\end{equation}
for all $\varphi\in V$ and $\sigma\in H^s_n(\Om)$. Moreover there exists a constant $C>0$ independent of $\eps\in(0,1)$ such that
\begin{align}\label{unibound}
 E^\eps(\varphi,\sigma)\geq C(\|\varphi\|^2+\|\sigma\|^2_{H^s_n(\Omega)}).
\end{align}
We consider the functional 
\begin{align}\label{E^0}
 E^0(\varphi,\sigma):=\begin{cases}
                       M^0(\varphi-\sigma)+F(\varphi,\sigma)&\text{if }(\varphi-\sigma)\in BV(\Om,\{-1,1\})\text{ and }\sigma\in H^s_n(\Om)\\
                       +\infty&\text{otherwise.}\\
                      \end{cases}
\end{align}

We now study the relation between $E^\eps$ and $E^0$.
\begin{theorem}\label{Gamma:conv}
 The functionals $E^\eps(\varphi,\sigma)$ \betti gamma \bl converge to the functional $E^0(\varphi,\sigma)$ with respect to the $L^1\times L^1$ topology.
\end{theorem}
\begin{proof}
 \textbf{Liminf inequality.} Let $(\varphi_\eps,\sigma_\eps)$ be a sequence converging to $(\varphi,\sigma)$ in $L^1(\Om)\times L^1(\Om)$. We must demonstrate that 
 \begin{equation}
  \liminf_{\eps\rightarrow0}E^\eps(\varphi_\eps,\sigma_\eps)\geq E^0(\varphi,\sigma).
 \end{equation}
So, let us assume the left-hand side being equal to a finite real number $M>0$. In particular condition \eqref{unibound} implies that, up to a subsequence, $\sigma^\eps\rightharpoonup\sigma$ weakly in $H^s_n(\Om)$. Moreover, from the boundedness $
 W(\varphi^\eps-\sigma^\eps)\leq \eps M$ and  the growth condition of $W$ we infer $\varphi^\eps-\sigma^\eps\rightarrow\varphi-\sigma$ strongly in $H$. This, together with the convergence of $\sigma^\eps$, provides $\varphi^\eps\rightarrow\varphi$ strongly in $H$.
Now, thanks to Theorem \ref{Modica}, we already know that 
\begin{equation}
  \liminf_{\eps\rightarrow0}M^\eps(\varphi_\eps-\sigma_\eps)\geq M^0(\varphi-\sigma).
 \end{equation}
 Thus the thesis follows from the semicontinuity inequality
 \begin{equation}
  \liminf_{\eps\rightarrow0}F(\varphi_\eps,\sigma_\eps)\geq F(\varphi,\sigma),
 \end{equation}
 that holds true thanks to the strong convergences in $H$ of both $\varphi^\eps$ and $\sigma^\eps$, and the weak convergence of $\sigma^\eps$ in $H^s_n(\Om)$.
 
 \textbf{Limsup inequality.} Let $(\varphi,\sigma)$ be such that $E^0(\varphi,\sigma)<+\infty$. By Theorem \ref{Modica} there exists a sequence $u^\eps\rightarrow u:=\varphi-\sigma$ such that 
\begin{equation}
  \limsup_{\eps\rightarrow0}M^\eps(u^\eps)\leq M^0(u).
 \end{equation}
 For all $\eps\in(0,1)$ we then set $\sigma^\eps:=\sigma$ and $\varphi^\eps:=u^\eps+\sigma^\eps$. Again by the coerciveness properties of $M^\eps$ we obtain $u^\eps\rightarrow u$ strongly in $H$, and we easily find out 
\begin{equation}
  \lim_{\eps\rightarrow0}F(\varphi^\eps,\sigma^\eps)=F(\varphi,\sigma),
 \end{equation}
 from which the thesis follows.
\end{proof}

\section{Existence of approximate gradient flow}\label{appr_gradient}
In this section we show the existence of solutions to the approximate gradient flow
\begin{align}
 (\dot\varphi^\eps,\dot\sigma^\eps)=-\nabla_{X\times Y}E(\varphi^\eps,\sigma^\eps).
\end{align}
We want to take the topologies  $X:=H^{-s}_n(\Om)$ and $Y:=H$\bl. The corresponding system of equations is
\begin{align}\label{gf:eps}
 \begin{cases}
  &\dot\varphi^\eps=-A^s\Big(\frac{1}{\eps}W'(\varphi^\eps-\sigma^\eps)-\eps\Delta(\varphi^\eps-\sigma^\eps)\Big)-A^s\sigma^\eps,\\
  &\dot\sigma^\eps=-\eps\Delta(\varphi^\eps-\sigma^\eps)+\frac{1}{\eps}W'(\varphi^\eps-\sigma^\eps)-A^s\sigma^\eps-\sigma^\eps-\varphi^\eps.
 \end{cases}
\end{align}
The existence theorem is stated as follows.
\begin{theorem}\label{exist:eps}
 Let $\eps\in(0,1)$ and $T>0$. For all $(\varphi^\eps_0,\sigma^\eps_0)\in H_n^1(\Om)\times H_n^s(\Om)$ there exist $(\varphi^\eps,\sigma^\eps)$ with
 \begin{align}
  &\varphi^\eps\in L^\infty(0,T; H^1_n(\Om)\bl)\cap H^1(0,T;H^{-s}_n(\Om)),\\
  &\sigma^\eps\in L^\infty(0,T;H_n^s(\Om))\cap H^1(0,T;H),
 \end{align}
satisfying \eqref{gf:eps} a.e. on $[0,T]$, with initial conditions $(\varphi^\eps(0),\sigma^\eps(0))=(\varphi^\eps_0,\sigma^\eps_0)$. Moreover we have
\begin{align}\label{+reg}
 \varphi^{\eps}\in L^2(0,T;H^2(\Om))\;\;\text{ and }\;\sigma^{\eps}\in L^2(0,T;H_n^{2s}(\Om)),\quad v^\eps\in L^2(0,T;H_n^{s}(\Om)),
\end{align}
where $v^\eps:=\frac{W'(u^\eps)}{\eps}-\eps\Delta u^\eps$.
\end{theorem}
\begin{proof}

In order to prove Theorem \eqref{exist:eps} we use a standard technique of discretization and an Euler implicit scheme. In the rest of the proof we will drop the label $\eps$ from the formulas.
 It is convenient to write the potential  $W(x)=(x^2-1)^2$ as the sum of a convex and a monotone part, namely $W:=\tilde W+\bar W$, with $\tilde W(x)=x^4+1$ and $\bar W(x)=-2x^2$.
We fix a time $T>0$. Let $n\in\mathbb N$ and define $\tau:=T/n$. Setting $\varphi_0:=\varphi^\eps(0)$ and $\sigma_0:=\sigma^\eps(0)$, we define recursively
\begin{align}
 \sigma_k:=\argmin_{\sigma\in H^s_n(\Omega)} &\frac{1}{\tau}\|\sigma-\sigma_{k-1}\|^2+\frac{1}{\eps}\int_{\Om}\tilde W(\varphi_{k-1}-\sigma)dx\nonumber\\
 &-\frac{1}{\eps}\int_{\Om}\bar W'(\varphi_{k-1}-\sigma_{k-1})\sigma dx+\frac{\eps}{2}\|\nabla(\varphi_{k-1}-\sigma)\|^2\\
 \nonumber
 &+\frac{a_s(\sigma,\sigma)}{2}+\frac{\|\sigma\|^2}{2}+\langle\varphi_{k-1},\sigma\rangle,\\
 \varphi_k:=\argmin_{\varphi\in H^1_n(\Omega)}&\frac{1}{\tau}\|\varphi-\varphi_{k-1}\|_{H_n^{-s}}^2+\frac{1}{\eps}\int_{\Om}\tilde W(\varphi-\sigma_k)dx+\frac{1}{\eps}\int_{\Om}\bar W'(\varphi_{k-1}-\sigma_{k})\varphi dx\nonumber\\
 &+\frac{\eps}{2}\|\nabla(\varphi-\sigma_k)\|^2+\langle\varphi,\sigma_{k}\rangle,
\end{align}
for $k=1,\dots,n$. Notice that the minimizers exist and are unique thanks to the convexity and coerciveness of the functionals. By minimality we get the two Euler conditions
\begin{align}
 &\langle \frac{\sigma_k-\sigma_{k-1}}{\tau},\psi_1\rangle-\langle\frac{1}{\eps}\tilde W'(\varphi_{k-1}-\sigma_{k}),\psi_1\rangle-\langle\frac{1}{\eps}\bar W'(\varphi_{k-1}-\sigma_{k-1}),\psi_1\rangle\nonumber\\
 &-\langle\eps\nabla(\varphi_{k-1}-\sigma_{k}),\nabla\psi_1\rangle+\langle\varphi_{k-1}+\sigma_{k},\psi_1\rangle+a_s(\sigma_k,\psi_1)=0,\label{eul1}\\
 &( \frac{\varphi_k-\varphi_{k-1}}{\tau},\psi_2)_{H^{-s}_n(\Omega)}+\langle\frac{1}{\eps}\tilde W'(\varphi_{k}-\sigma_{k}),\psi_2\rangle\nonumber\\&+\langle\frac{1}{\eps}\bar W'(\varphi_{k-1}-\sigma_{k}),\psi_2\rangle+\langle{\eps}\nabla(\varphi_{k}-\sigma_{k}),\nabla\psi_2\rangle+\langle\sigma_k,\psi_2\rangle=0,\label{eul2}
\end{align}
valid for all $\psi_1\in H_n^s(\Om)$ and $\psi_2\in H_n^1(\Om)$.
Testing \eqref{eul1} by $\psi_1=\tau^{-1}(\sigma_k-\sigma_{k-1})$ and \eqref{eul2} by $\psi_2=\tau^{-1}(\varphi_k-\varphi_{k-1})$, then summing the two expressions and using the inequalities
\begin{align*}
&-\big(\tilde W'(\varphi_{k-1}-\sigma_{k})+\bar W'(\varphi_{k-1}-\sigma_{k-1})\big)(\sigma_k-\sigma_{k-1}) \geq W(\varphi_{k-1}-\sigma_{k})\\
\nonumber
&\qquad\qquad\qquad\qquad\qquad\quad\qquad\qquad\qquad\qquad\qquad\qquad-W(\varphi_{k-1}-\sigma_{k-1}),\\
&\big(\tilde W'(\varphi_{k}-\sigma_{k})+\bar W'(\varphi_{k-1}-\sigma_{k})\big)(\varphi_k-\varphi_{k-1}) \geq W(\varphi_{k}-\sigma_{k})-W(\varphi_{k-1}-\sigma_{k}),
\end{align*}
we get
\begin{align}
 &\|\frac{\sigma_k-\sigma_{k-1}}{\tau}\|^2+\|\frac{\varphi_k-\varphi_{k-1}}{\tau}\|_{H^{-s}_n(\Omega)}^2+\frac{1}{\tau\eps}\int_\Om W(\varphi_{k}-\sigma_{k})dx-\frac{1}{\tau\eps}\int_\Om W(\varphi_{k-1}-\sigma_{k-1})dx \nonumber\\
 &-\langle{\eps}\nabla(\varphi_{k-1}-\sigma_{k}),\frac{\nabla\sigma_k-\nabla\sigma_{k-1}}{\tau}\rangle+\langle{\eps}\nabla(\varphi_{k}-\sigma_{k}),\frac{\nabla\varphi_k-\nabla\varphi_{k-1}}{\tau}\rangle\nonumber\\
 &+\langle\varphi_{k-1}+\sigma_{k},\frac{\sigma_k-\sigma_{k-1}}{\tau}\rangle+\langle\nabla\sigma_k,\frac{\nabla\sigma_k-\nabla\sigma_{k-1}}{\tau}\rangle+\langle\sigma_k,\frac{\varphi_k-\varphi_{k-1}}{\tau}\rangle\nonumber\\
 &+a_s(\sigma_k,\frac{\sigma_k-\sigma_{k-1}}{\tau})\leq0.\label{est:en}
\end{align}
We define the following piecewise affine and constant functions on $[0,T]$
\begin{align}
 &\sigma_\tau(t):=\sigma_{k-1}+\tau^{-1}(t-t_k)(\sigma_k-\sigma_{k-1})\;\;\;\text{ for }t\in[t_{k-1},t_k),\nonumber\\
 &\varphi_\tau(t):=\varphi_{k-1}+\tau^{-1}(t-t_k)(\varphi_k-\varphi_{k-1})\;\;\;\text{ for }t\in[t_{k-1},t_k),\nonumber\\
&\tilde\sigma_\tau(t):=\sigma_{k-1}\;\;\;\text{ for }t\in[t_{k-1},t_k),\nonumber\\
 &\tilde\varphi_\tau(t):=\varphi_{k-1}\;\;\;\text{ for }t\in[t_{k-1},t_k).
 \end{align}
For all $t\in[t_{k-1},t_k)$ it holds
\begin{align}
 \langle\nabla\sigma_k,\frac{\nabla\sigma_k-\nabla\sigma_{k-1}}{\tau}\rangle=\langle\nabla\sigma_\tau(t),\nabla\dot\sigma_\tau(t)\rangle+|t_k-t|\|\nabla\dot\sigma_\tau(t)\|^2,
\end{align}
implying
\begin{align}\label{comp2}
 \int_{t_{k-1}}^{t_k}\langle\nabla\sigma_k,\frac{\nabla\sigma_k-\nabla\sigma_{k-1}}{\tau}\rangle\geq\frac{1}{2}\|\nabla\sigma_k\|^2-\frac{1}{2}\|\nabla\sigma_{k-1}\|^2.
\end{align}
Similarly, we get the estimates
\begin{align}\label{comp3}
 \int_{t_{k-1}}^{t_k}a_s(\sigma_k,\frac{\sigma_k-\sigma_{k-1}}{\tau})\geq\frac{1}{2}a_s(\sigma_k,\sigma_k)-\frac{1}{2}a_s(\sigma_{k-1},\sigma_{k-1}),
\end{align}
and 
\begin{align}\label{comp4}
 \int_{t_{k-1}}^{t_k}\langle\sigma_k,\frac{\sigma_k-\sigma_{k-1}}{\tau}\rangle\geq\frac{1}{2}\|\sigma_k\|^2-\frac{1}{2}\|\sigma_{k-1}\|^2.
\end{align}
Moreover, for $t\in[t_{k-1},t_k)$,
\begin{align*}
 &-\langle\eps\nabla(\varphi_{k-1}-\sigma_{k}),\frac{\nabla\sigma_k-\nabla\sigma_{k-1}}{\tau}\rangle+\langle\eps\nabla(\varphi_{k}-\sigma_{k}),\frac{\nabla\varphi_k-\nabla\varphi_{k-1}}{\tau}\rangle\nonumber\\
 &=\langle\eps\nabla(\varphi_{k}-\sigma_{k}),\nabla\dot\varphi_\tau-\nabla\dot\sigma_\tau\rangle+\tau\eps\langle\nabla\dot\varphi_\tau,\nabla\dot\sigma_\tau\rangle\\
 &=\langle\eps\nabla(\varphi_{\tau}-\sigma_{\tau}),\nabla\dot\varphi_\tau-\nabla\dot\sigma_\tau\rangle+\eps(t_k-t)\|\nabla(\varphi_{\tau}-\sigma_{\tau})\|^2+\tau\eps\langle\nabla\dot\varphi_\tau,\nabla\dot\sigma_\tau\rangle,
\end{align*}
so that integrating on $[t_{k-1},t_k)$ we arrive at
\begin{align}\label{comp}
&\frac{\eps}{2}\|\nabla\varphi_k-\nabla\sigma_k\|^2-\frac{\eps}{2}\|\nabla\varphi_{k-1}-\nabla\sigma_{k-1}\|^2+\frac{\eps\tau^2}{2}\|\nabla\dot\varphi_\tau-\nabla\dot\sigma_\tau\|^2+\eps\tau^2\langle\nabla\dot\varphi_\tau,\nabla\dot\sigma_\tau\rangle\nonumber\\
&=\frac{\eps}{2}\|\nabla\varphi_k-\nabla\sigma_k\|^2-\frac{\eps}{2}\|\nabla\varphi_{k-1}-\nabla\sigma_{k-1}\|^2+\frac{\eps\tau^2}{2}\|\nabla\dot\varphi_\tau\|^2+\frac{\eps\tau^2}{2}\|\nabla\dot\sigma_\tau\|^2.
\end{align}
Then, for $0<K<n$ we integrate over $[t_{k-1},t_k)$ expression \eqref{est:en} and sum over $k=1,\dots,K$, so that,  
taking into account \eqref{comp2}, \eqref{comp3}, \eqref{comp4}, and \eqref{comp}, we infer
\begin{align}\label{gronw}
 &E^\eps(\varphi_\tau(t_K),\sigma_\tau(t_K))+\int_{0}^{t_K}\|\dot\sigma_\tau(t)\|^2+\|\dot\varphi_\tau(t)\|_{H^{-s}_n(\Om)}^2dt\leq E^\eps(\varphi_\tau(0),\sigma_\tau(0)).
\end{align}
This, together with the coerciveness property \eqref{coerc}, implies that there exists a constant $M>0$ independent of $\tau$ such that
\begin{subequations}
\begin{align}
&\|\varphi_\tau\|_{L^\infty(0,T;H_n^1(\Om))}\leq M,\\
&\|\sigma_\tau\|_{L^\infty(0,T;H^s_n(\Om))}\leq M,\\
 &\|\dot\varphi_\tau\|_{L^2(0,T;{H^{-s}_n(\Om)})}\leq M,\\
 &\|\dot\sigma_\tau\|_{L^2(0,T;H)}\leq M.
\end{align}
 \end{subequations}
Therefore we find $\varphi\in L^\infty(0,T;H_n^1(\Om))\cap H^1(0,T;{H^{-s}_n(\Om)})$ and $\sigma\in L^\infty(0,T;H_n^s(\Om))\cap H^1(0,T;H)$ such that, up to a subsequence,
\begin{subequations}\label{conv1}
\begin{align}
& \sigma_\tau\rightharpoonup \sigma\text{ weakly* in } L^\infty(0,T;H_n^s(\Om)),\\
& \varphi_\tau\rightharpoonup \varphi\text{ weakly* in } L^\infty(0,T;H_n^1(\Om)),
\end{align}
\end{subequations}
and
\begin{align}
 &\sigma_\tau\rightharpoonup \sigma\text{ weakly in } H^1(0,T;H),\label{conv3}\\
 &\varphi_\tau\rightharpoonup \varphi\text{ weakly in } H^1(0,T;{H^{-s}_n(\Om)}).\label{conv4}
\end{align}
From \eqref{conv1} we infer
\begin{equation}\label{conv5}
 W'(\varphi_\tau-\sigma_\tau)\rightharpoonup W'(\varphi-\sigma)\text{ weakly* in } L^\infty(0,T;H),
\end{equation}
thanks to the cubic growth of $W'$ and of the Sobolev embedding $V\subset L^6(\Om)$.
Therefore, multiplying \eqref{eul1} and \eqref{eul2} by an arbitrary test function $g\in C(0,T;\mathbb R)$ and integrating on time, we can pass to the limit in $\tau\rightarrow0$ thanks to \eqref{conv1}-\eqref{conv5}, and then the arbitrariness of $g$ entails that, almost everywhere on $[0,T]$,
\begin{align}
 &\langle \dot\sigma,\psi_1\rangle=\langle\frac{1}{\eps}W'(\varphi-\sigma),\psi_1\rangle+\langle{\eps}\nabla(\varphi-\sigma),\nabla\psi_1\rangle-\langle\varphi+\sigma,\psi_1\rangle-a_s(\sigma,\psi_1),\label{eul1:lim}
 \end{align}
 and  
 \begin{align}
 &\langle A^{-s}\dot\varphi,\psi_2\rangle=-\langle\frac{1}{\eps}W'(\varphi-\sigma),\psi_2\rangle-\langle{\eps}\nabla(\varphi-\sigma),\nabla\psi_2\rangle-\langle\sigma,\psi_2\rangle,\label{eul2:lim}
\end{align}
for all $\psi_1\in H_n^s(\Om)$ and $\psi_2\in H_1^s(\Om)$.
From \eqref{eul2:lim} we infer, almost everywhere on $[0,T]$,
\begin{align}\label{4.24}
 &A^{-s}\dot\varphi=-v-\sigma,
\end{align}
where we have set $$v:=\frac{1}{\eps}W'(\varphi-\sigma)-{\eps}\Delta(\varphi-\sigma).$$ Thanks to \eqref{conv3} and \eqref{conv5}, expression \eqref{4.24} implies 
\begin{align}\label{reg1}
 v\in L^2(0,T;H_n^s(\Om))\;\;\text{ and }\;\Delta(\varphi-\sigma)\in L^2(0,T;H)).
\end{align}
On the other hand, comparing all the terms in \eqref{eul1:lim}, we also find 
\begin{align}\label{reg2}
 A^s \sigma\in L^2(0,T;H),
\end{align}
which, together with \eqref{reg1}, implies
\begin{align}
 \varphi\in L^2(0,T;H^2(\Om))\;\;\text{ and }\;\sigma\in L^2(0,T;H_n^{2s}(\Om)).
\end{align}
 
\end{proof}

\section{The sharp interface limit}\label{sharp}
To prove the convergence of the gradient flow considered in the previous section, we have to make some necessary hypotheses on our setting and on the regularity of our solutions.
\subsection{Hypotheses}\label{hypotheses}
 \begin{itemize}
   \item[(HP1)] We assume $s>3/2$. 
 \item[(HP2)] We assume that the functions $u^\eps(t):=\varphi^\eps(t)-\sigma^\eps(t)$ are of class $C^3(\bar\Omega)$ \bl and the limiting interface  $\Gamma$ \bl is of class $C^3$ in the time-space $[0,T^*]\times \Om$ for some $T^*\in(0,T]$.
 \end{itemize}
 
 In the case $H^s(\Om)\subset W^{1,p}(\Om)$ for some $p>d=3$ (which holds true if $s>3/2$, namely under hypothesis (HP1)), since $v^\eps=\frac{W'(u^\eps)}{\eps}-\eps\Delta u^\eps\in L^2(0,T;H^s_n(\Om))$, then \bl the following convergence  in the sense of Radon measures holds  true\bl 
  \begin{align}\label{ipotesi}
   \frac{\eps}{2}|\nabla u^\eps|^2+\frac{W(u^\eps)}{\eps}\rightharpoonup 2c_Wd\mathcal H^{d-1}\llcorner_\Gamma.
  \end{align}
 Thanks to (HP1) this follows from \cite[Theorem 3.2]{RoTon}.
 The last property  implies the following fact, called \textit{equipartition of energy}: in the sense of Radon measures
  \begin{align}
   \Big|\frac{\eps}{2}|\nabla u^\eps|^2-\frac{W(u^\eps)}{\eps}\Big|\rightharpoonup 0.
  \end{align}
  This is proved in \cite[Lemma 1]{LuMo}.
 In \cite{Ton1,Ton2} it is shown that without the restriction $s>3/2$,  formula \eqref{ipotesi} holds with an additional factor $\theta$ at the right-hand side, which is a positive integer-valued function supported on $\Gamma$. Since it might be a non-constant function, hypothesis (HP1) turns out to be necessary. It is conjectured that \eqref{ipotesi} must hold also in the case $s=1$, but this issue is still an open question (see \cite{RoTon}). In particular, in the case that $s\in[1,3/2]$ we have to make the alternative hypothesis
 \begin{itemize}
  \item[(HP1bis)] We assume  $s\in[1,3/2]$ and that \eqref{ipotesi} holds true.
 \end{itemize}
In other words, we assume that the integer-valued function $\theta$ supported on $\Gamma$ is constantly equal to $1$, as conjectured in \cite{RoTon}.
\newline

Let us now consider the approximate gradient flows: for all $\eps>0$ let $(\varphi^\eps,\sigma^\eps)$ be a solution to system \eqref{gf:eps} with initial data $(\varphi_0^\eps,\sigma_0^\eps)$, as provided by Theorem \ref{exist:eps}. Let $(\varphi_0,\sigma_0)$ be functions in $H^{2}(\Om)\cap H^{1}_n(\Om)\times H^{2s}(\Om)\cap H^{s}_n(\Om)$ such that $E^0(\varphi_0,\sigma_0)<+\infty$. In particular, such condition requires that $\varphi_0-\sigma_0=u_0=\chi_{\Om_0^+}-\chi_{\Om_0^-} $, where $\Om_0^+$ and $\Om_0^-$ are two disjoint sets with $\Om_0^-=\Om\setminus\Om_0^+$ and common boundary $\Gamma_0$. We assume that the interface $\Gamma_0$ is of class $C^3$, according with hypothesis (HP2).

We will assume that the initial conditions are well prepared, that is
\begin{align}\label{wellprepared}
&(\varphi_0^\eps,\sigma_0^\eps)\rightarrow (\varphi_0,\sigma_0)\;\;\;\text{ strongly in }L^1(\Om)\times L^1(\Om),\nonumber\\
 &\text{and }\;\;E^\eps(\varphi_0^\eps,\sigma_0^\eps)\rightarrow E^0(\varphi_0,\sigma_0).
\end{align}
In other words, we are assuming that $(\varphi_0^\eps,\sigma_0^\eps)$ is a recovery sequence for $(\varphi_0,\sigma_0)$.
 
Thanks to \cite[Theorem 1.2B]{Le} it is possible to construct a sequence of well-prepared initial data. We address to the following theorem  the easy adaptation to our case.

\begin{theorem}\label{well_p}
  Let $(\varphi_0,\sigma_0)\in L^1(\Om)\times H^{2s}(\Om)$ such that $\varphi_0-\sigma_0=u_0=1-2\chi_{\Om_0^-}=2\chi_{\Om_0^+}-1$, with $\Om^-_{\betti 0}\subset\Om$ having boundary $\Gamma_0=\partial\Om_0^-$ a $C^3$-closed surface in $\Om$ with finite $\mathcal H^2$ measure. 
  Then there exists a sequence of smooth functions $(\varphi^\eps_0,\sigma^\eps_0)$ such that \eqref{wellprepared} holds true.
 \end{theorem}
 \begin{proof}
  Theorem 1.2 of \cite{Le} ensures the existence of a sequence $u^\eps_0$ of smooth functions such that $u^\eps_0\rightarrow u_0$ strongly in $L^1(\Om)$. Since the surface $\Gamma_0$ is closed in $\Om$, it has strictly positive distance from the boundary.  Assume the inner part of $\Gamma_0$ being $\Om_0^+$. Using a suitable cut-off function $\zeta$ which equals $1$ on $\bar\Om_0^+$ and $0$ on a neighborhood of $\partial \Om$, and then replacing $u_0^\eps$ by $\zeta(u^\eps_0+1)-1$, it is not restrictive to assume that $u^\eps_0\in H^1_n(\Om)\cap H^s_n(\Om)$. In order to have \eqref{wellprepared} it suffices to choose a sequence $\sigma^\eps_0\in H^s_n(\Om)$ converging strongly to $\sigma_0$ in $H^{2s}_n(\Om)$.  Then setting $\varphi^\eps_0:=\sigma^\eps_0+u^\eps_0$ the thesis easily follows thanks to the form of the energy $E_0$.
 \end{proof}

 \subsection{Convergence of gradient flows}
%

Let us start with the following statement:
\begin{lemma}\label{energy:velocity}
 Let $\cup_{t\in[0,T^*]}\Gamma(t)\times\{t\}\subset \Om\times[0,T^*]$ be a $C^3$ hypersurface with $\Gamma(t)$ closed for all $t\in[0,T^*]$. Let $u(t):=\chi_{\Om^+(t)}-\chi_{\Om^{\betti -}(t)}$ for all $t\in[0,T^*]$, and assume $u\in L^\infty(0,T^*;H)\cap H^1(0,T^*;H_n^{-s}(\Om))$ and  $\sigma\in L^\infty  (0,T^*;H_n^{s}(\Om))\cap L^2 (0,T^*;H_n^{2s}(\Om))\cap H^1(0,T^*;H)$. Then, for a.e. $t\in[0,T^*]$, 
 \begin{align*}
  \frac{d}{dt}E^0(u(t)+\sigma(t),\sigma(t))=&-2c_W(V(t),k(t))_{L^2(\Gamma)}+2(V(t),\sigma(t))_{L^2(\Gamma)}\\
  \nonumber
  &+(\dot\sigma(t),A^s \sigma(t)+u(t)+3\sigma(t)),
 \end{align*}
where $V(t)=\dot\Gamma(t)$ is the normal velocity of the interface  $\Gamma(t)$, and $k(t)$ is its mean curvature.
\end{lemma}
\begin{proof}
 Using \cite[Theorem 7.31]{AFP} (see also \cite[formula (2.4)]{Le}) we obtain
 $$\frac{d}{dt}(c_W\HI^1(\Gamma))(t)=-2c_W( \dot\Gamma(t),k(t))_{L^2(\Gamma)}.$$
The time derivative of $\langle u+\sigma,\sigma\rangle+\frac{1}{2}\|A^{s/2}\sigma\|^2+\frac{1}{2}\|\sigma\|^2$ instead reads
\begin{align*}
 &\langle \dot u(t),\sigma\rangle+\langle\dot\sigma(t), u(t)\rangle+3(\dot\sigma(t),\sigma(t))+( \dot\sigma(t),A^{s}\sigma(t))=\\
 &=2( \dot\Gamma(t),\sigma)_{L^2(\Gamma)}+(\dot\sigma(t),A^{s}\sigma(t)+ u(t)+3\sigma(t)).
\end{align*}
\end{proof}

It is convenient to denote the normal velocity of the interface by $\dot \Gamma$, so that we write $\dot u=2\dot \Gamma$.

Let us recall that the functions $(\varphi^\eps,\sigma^\eps)$ satisfy
\begin{align}
 \begin{cases}
  &\dot\varphi^\eps=-A^sv^\eps-A^s\sigma^\eps,\label{gf:eps_bis}\\
  &\dot\sigma^\eps=v^\eps-A^s\sigma^\eps-\sigma^\eps-\varphi^\eps,
 \end{cases}
\end{align}
with the corresponding  energy balance
\begin{align}\label{energy}
 E^\eps(\varphi_0^\eps,\sigma_0^\eps)-E^\eps(\varphi^\eps(t),\sigma^\eps(t))=\int_0^t\|\dot\varphi^\eps\|_{H^{-s}_n(\Om)}^2+\|\dot\sigma^\eps\|^2ds,
\end{align}
valid for all $t\in[0,T]$. From this we easily infer some a-priori estimates.

\begin{prop}
For $\eps\in(0,1)$ let $(\varphi^\eps,\sigma^\eps)$ be a solution in Theorem \ref{exist:eps} with initial datum $(\varphi_0^\eps,\sigma_0^\eps)$. Then there exists a constant $M>0$ such that
\begin{align}
 &\|\varphi^\eps\|_{L^\infty(0,T;H)}\leq M,\label{apriori1}\\
 &\|\sigma^\eps\|_{L^\infty(0,T;H^s_n(\Om))}\leq M,\label{apriori2}\\
 &\|\varphi^\eps\|_{H^1(0,T;H^{-s}_n(\Om))}\leq M,\label{apriori3}\\
 &\|\sigma^\eps\|_{H^1(0,T;H)}\leq M,\label{apriori4}
\end{align}
and, setting $v^\eps:= \frac{1}{\eps}W'(\varphi^\eps-\sigma^\eps)-\eps\Delta(\varphi^\eps-\sigma^\eps)$ and $u^\eps:=\varphi^\eps-\sigma^\eps$, we have
\begin{align}
 &\|v^\eps\|_{L^2(0,T;H_n^s(\Om))}\leq M,\label{apriori5}\\
 &\|\sigma^\eps\|_{L^2(0,T;H_n^{2s}(\Om))}\leq M,\label{apriori6}\\
 &\|u^\eps\|_{L^\infty(0,T;L^4(\Om))}\leq M\label{apriori7}.
\end{align}
for all $\eps\in (0,1)$.
\end{prop}
\begin{proof}

For all $t\in[0,T]$ we have
\begin{align}
 \int_0^t\|\dot\sigma^\eps\|^2+\|\dot\varphi^\eps\|_{H^{-s}_n(\Om)}^2ds+E(\varphi^\eps(t),\sigma^\eps(t))=E(\varphi_0^\eps,\sigma_0^\eps)\leq M,
\end{align}
which, together with \eqref{unibound}, implies \eqref{apriori1}-\eqref{apriori4}.
The uniform boundedness \eqref{unibound} and the coerciveness property $E^\eps(\varphi,\sigma)\geq C\|\varphi-\sigma\|_{L^4}^4(\Om)$ imply
\begin{align}
 \|u^\eps\|_{L^\infty(0,T;L^4(\Om))}\leq M.
\end{align}
\betti
Subtracting the two equations in \eqref{gf:eps_bis}, we obtain
$$A^sv^\eps+v^\eps=\varphi^\eps+\sigma^\eps-\dot\varphi^\eps+\dot\sigma^\eps,$$
so that, by \eqref{apriori3} and \eqref{apriori4}, we see that $$\|A^sv^\eps+v^\eps\|_{L^2(0,T;H^{-s}_n(\Om))}\leq M,$$ for all $\eps\in(0,1)$. Therefore we infer \eqref{apriori5} from the fact that the bilinear form $a_s(\cdot,\cdot)+(\cdot,\cdot)_H$ is coercive on $H^s_n(\Om)$.
From this, \eqref{apriori1}, \eqref{apriori2}, and \eqref{apriori4}, thanks to the fact that $\dot \sigma^\eps=-A^s\sigma^\eps+v^\eps-\varphi^\eps-\sigma^\eps $, we infer
\begin{align}
 \|A^s \sigma^\eps\|_{L^2(0,T;H)}\leq M,
\end{align}
 implying \eqref{apriori6}.
\end{proof}
\betti
\begin{remark}
 Notice that estimate \eqref{apriori1} can be refined. Actually from \eqref{apriori2}, \eqref{apriori7}, and the embedding $ H^s_n(\Om)\subset L^4(\Om)$ we obtain
 \begin{align}
 \|\varphi^\eps\|_{L^\infty(0,T;L^4(\Om))}\leq M.  
 \end{align}
\end{remark}
\bl

\begin{prop}\label{u:conv}
 For a subsequence, we have
 \begin{align}
 u^\eps\rightharpoonup u \;\;\;\;\text{ weakly in }L^4(\Om\times[0,T]).
 \end{align}
 Moreover, for all $t\in[0,T]$, $u(t)\in BV(\Om;\{-1,1\})$ and 
\begin{align}
& u^\eps(t)\rightharpoonup u(t)\;\;\;\;\text{ weakly in }L^4(\Om),\\
& u^\eps(t)\rightarrow u(t)\;\;\;\;\text{ strongly in }L^1(\Om),\\
& u^\eps(t)\rightharpoonup u(t)\;\;\;\;\text{ weakly* in }BV(\Om). 
\end{align}
\end{prop}

\begin{proof}
The first statement readily follows from \eqref{apriori7}. To conclude the proof it suffices to follow the lines of the proof of  \cite[Proposition 4.1]{Le}, which can be trivially adapted to our case.
\end{proof}

\begin{prop}\label{prop:v=k}
Up to a subsequence, the functions $v^\eps\rightharpoonup v$ weakly in $L^2(0,T;V)$ and the limit function $v$ satisfies for a.e. $t\in[0,T]$
 \begin{equation}\label{v=mk}
  v(t)=-c_Wk(t)\;\;\;\text{ on }\Gamma(t),
 \end{equation}
where $k(t)\in H^{-1/2}(\Gamma(t))$ is the mean curvature of the smooth surface $\Gamma(t)$ at time $t$. 
\end{prop}
\begin{proof}
 To prove this we argue as in \cite[Lemma 3.1]{Le}. Note that this result is strongly based on hypothesis (HP1) or (HP1bis).
\end{proof}

\begin{prop}
 For all $t\in[0,T]$ there holds
 \begin{equation}\label{velocities:1}
  \liminf_{\eps\rightarrow0}\int_0^t\|\dot\varphi^\eps(s)\|_{H_n^{-s}(\Om)}ds\geq \int_0^t\|2\dot\Gamma(s)+\dot\sigma(s)\|_{H^{-s}_n(\Om)}ds.
 \end{equation}

\end{prop}
\begin{proof}
 Since $u^\eps\rightarrow u$ in $L^1([0,T]\times\Om)$ we know that $\dot u^\eps$ tends to $\dot u=2\dot\Gamma$ in the sense of distributions. On the other hand, we know that $\dot u^\eps=(\dot\varphi^\eps-\dot \sigma^\eps)\rightharpoonup(\dot \varphi-\dot \sigma)$ weakly in $L^2(0,T;H^{-s}_n(\Om))$, so that $\dot \varphi-\dot \sigma=2\dot\Gamma\in L^2(0,T;H^{-s}_n(\Om))$. This, together with \eqref{apriori4}, implies 
 $$\dot\varphi^\eps\rightharpoonup 2\dot\Gamma+\dot\sigma\;\;\;\;\text{ weakly in }L^2(0,T;H^{-s}_n(\Om)).$$
The thesis then follows by lower semicontinuity.
 \end{proof}

  Now we are ready to state the main result of the paper. \bl

 \begin{theorem}\label{main}
 Let us assume hypotheses (HP1) or (HP1bis), and (HP2). Suppose that the initial data satisfy \eqref{wellprepared}. Then there exists a time $T^*\in (0,T]$ such that it holds
  \begin{align}\label{f1}
  &2\dot\Gamma(t)=-A^s v(t)+\varphi(t)+\sigma(t)-v(t),\\
  \label{f2}
  &\dot\sigma(t)=-A^s\sigma(t)+v(t)-\varphi(t)-\sigma(t),
  \end{align}
 and
\begin{align}
&v(t)=-c_Wk \label{part:v}\;\;\;\mathcal H^2-\text{a.e. on }\Gamma,
\end{align}
for a.e. $t\in [0,T^*]$. Moreover it holds
\begin{align}\label{part:Vbis}
 A^s v(t)=\varphi(t)+\sigma(t)-v(t),
\end{align}
almost everywhere in $\Om^+\cup\Om^-$, and for a.e. $t\in[0,T^*]$.
\betti
Finally 
\begin{align}
 &\sigma^\eps(t)\rightarrow\sigma(t)\;\;\;\;\text{ strongly in }H^s_n(\Om),
 \end{align}
 for all $t\in[0,T^*]$ and
 \begin{align}
 &v^\eps\rightarrow v\;\;\;\;\text{ strongly in }L^2(0,T^*;H^{s}_n(\Om)).
\end{align}
\end{theorem}
\bl
 \begin{proof}
The energy identity \eqref{energy} together with \eqref{gf:eps} imply
\begin{align*}
  2E^\eps(\varphi_0^\eps,\sigma_0^\eps)-2E^\eps(\varphi^\eps(t),\sigma^\eps(t))=&\int_0^t\left(\|\dot\varphi^\eps\|_{H^{-s}_n(\Om)}^2+\|\dot\sigma^\eps\|^2\right)ds\\
 \nonumber
 &+\int_0^t\left(\|-A^sv^\eps-A^s\sigma^\eps\|_{H^{-s}_n(\Om)}^2\right.\\
 \nonumber
 &\left.\qquad+\|-A^s \sigma^\eps+v^\eps-\varphi^\eps-\sigma^\eps\|^2\right)ds.
\end{align*}
Thus taking the liminf as $\eps\rightarrow0$ we infer
\begin{align}\label{inequalities}
 &\liminf_{\eps\rightarrow0}E^\eps(\varphi_0^\eps,\sigma_0^\eps)-E^\eps(\varphi^\eps(t),\sigma^\eps(t))\nonumber\\
 &\geq \frac{1}{2}\int_0^t\|2\dot\Gamma+\dot\sigma\|_{{H^{-s}_n(\Om)}}^2+\|-A^sv-A^s\sigma\|_{H^{-s}_n(\Om)}^2 ds\nonumber\\
  &\qquad+\frac{1}{2}\int_0^t\|\dot\sigma\|^2+\|-A^s \sigma+v-\varphi-\sigma\|^2 ds\nonumber\\
 &\geq \int_0^t-( 2\dot\Gamma +\dot\sigma,A^{s}(v+\sigma))_{{H^{-s}_n(\Om)}} +(\dot\sigma,-A^s \sigma+v-\varphi-\sigma) ds\nonumber\\
 &= \int_0^t-2\langle \dot\Gamma,v+\sigma\rangle+(\dot\sigma,-A^s\sigma-u-3\sigma) ds\nonumber\\
&= \int_0^t-2(\dot\Gamma,v+\sigma)_{L^2(\Gamma)}+(\dot\sigma,-A^s \sigma-u-3\sigma) ds\nonumber\\
 &= \int_0^t2c_W(\dot\Gamma,k)_{L^2(\Gamma)}-2( \dot\Gamma,\sigma)_{L^2(\Gamma)}+(\dot\sigma,-A^s \sigma-u-3\sigma) ds\nonumber\\
 &=E(\varphi_0,\sigma_0)-E(\varphi(t),\sigma(t)).
\end{align}
We have used \eqref{velocities:1} and the lower semicontinuity of the  norms in $L^2(0,T;H^{-s}_n(\Om))$ and $L^2(0,T;H)$  in the first inequality, the Cauchy-Schwartz inequality in the second one, the identity \eqref{v=mk} in the second equality and Lemma \ref{energy:velocity} in the last one.
On the other hand, by Proposition \ref{u:conv} we have
$$(\varphi^\eps(t),\sigma^\eps(t))\rightarrow(\varphi(t),\sigma(t))\;\;\;\;\;\text{ in  }L^1(\Om)\times L^1(\Om),$$
so that Theorem \ref{Gamma:conv} and the hypothesis that $(\varphi^\eps_0,\sigma^\eps_0)$ is a recovery sequence entails
$$\limsup_{\eps\rightarrow0}E(\varphi^\eps_0,\sigma^\eps_0)-E(\varphi^\eps(t),\sigma^\eps(t))\leq E(\varphi_0,\sigma_0)-E(\varphi(t),\sigma(t)).$$
Therefore all the inequalities in \eqref{inequalities} are equalities, and in particular we get that for a.e. $t\in[0,T]$,
\begin{subequations}
\begin{align}
& 2\dot\Gamma(t) +\dot\sigma(t)=-A^{s}( v(t)+ \sigma(t))\;\;\;\;\text{ a.e. on }\Omega,\\
&\dot\sigma(t)=-A^s \sigma(t)+v(t)-\varphi(t)-\sigma(t)\;\;\;\;\text{ a.e. on }\Omega.
\end{align}
 \end{subequations}
Combining these two equations we infer
\begin{align}\label{flowlim3}
 2\dot\Gamma(t)=-A^s v(t)+\varphi(t)+\sigma(t)-v(t)
\end{align}
for a.e. $t\in[0,T]$ and hence \eqref{f1} and \eqref{f2} are proved. \bl Using the fact that $\dot\Gamma(t)$ is supported on $\Gamma$, using test functions in $C^\infty_c(\Om^+\cup\Om^-)$ it is easily seen that 
\begin{align}\label{Delta:v}
 A^s v(t)=\varphi(t)+\sigma(t)-v(t)\;\;\;\;\text{ on }\Om^+\cup\Om^-,
\end{align}
for a.e. $t\in[0,T]$. 
%
%
This is  \eqref{part:Vbis}\bl, and  \eqref{part:v} follows by Proposition~\ref{prop:v=k}\bl.
\betti
Finally we have seen that, for all $t\in[0,T^*]$
$$E^\eps(\varphi^\eps(t),\sigma^\eps(t))\rightarrow E(\varphi(t),\sigma(t)).$$
Moreover, \eqref{apriori2} and \eqref{apriori4} imply that for all $t\in[0,T^*]$
$$\sigma^\eps(t)\rightarrow\sigma(t)\;\;\;\;\text{ strongly in }H,$$
while 
$$\varphi^\eps(t)\rightharpoonup\varphi(t)\;\;\;\;\text{ weakly in }H,$$
by \eqref{apriori1} and \eqref{apriori3}. Hence we deduce, thanks to the special form of the energies \eqref{E^eps} and \eqref{E^0},
\begin{align}
 &\int_\Om\eps^{-1}W(u^\eps(t))+\frac{\eps}{2}|\nabla u^\eps(t)|^2dx\rightarrow c_W\mathcal H^2(\Gamma(t)),\\
 &a_s(\sigma^\eps(t),\sigma^\eps(t))\rightarrow a_s(\sigma(t),\sigma(t)),
\end{align}
and then 
\begin{align}
 \sigma^\eps(t)\rightarrow\sigma(t)\;\;\;\;\text{ strongly in }H^s_n(\Om).
\end{align}
Moreover, we have also seen that 
$$\int_0^t\|-A^sv^\eps-A^s\sigma^\eps\|_{H^{-s}_n(\Om)}^2 ds\rightarrow\int_0^t\|-A^sv-A^s\sigma\|_{H^{-s}_n(\Om)}^2 ds,$$
so that 
\begin{align}
 -A^sv^\eps\rightarrow-A^sv\;\;\;\;\text{ strongly in }L^2(0,T^*;H^{-s}_n(\Om)),
\end{align}
which implies 
\begin{align}
 v^\eps\rightarrow v\;\;\;\;\text{ strongly in }L^2(0,T^*;H^{s}_n(\Om)).
\end{align}

\end{proof}
\bl

In the particular case $s=1$ or $s=2$ we can deduce then from \eqref{f1} a condition relating $v$ and $\dot\Gamma$ on $\Gamma$. This relation is stated in the following  result.\bl 
\begin{theorem}\label{main2}
 Assume hypotheses of Theorem \ref{main} with (HP1bis) and  $s=1$. Then the additional condition holds true
 $$\left[\frac{\partial v}{\partial n}\right](t)=-2\dot\Gamma(t)\;\;\;\;\mathcal H^2-\text{a.e. on }\Gamma,$$
 for a.e. $t\in[0,T^*]$.
\end{theorem}
\begin{proof}
 Let us denote $w:=-v+A^{-1}\varphi+A^{-1}\sigma-A^{-1}v$. Equation \eqref{flowlim3} reads
 $Aw=2\dot\Gamma.$ This means that $w\in V_\Gamma$ and, using Lemma \ref{lemma:invDelta1}, that $-\Delta_\Gamma(w\llcorner_\Gamma)=2\dot\Gamma$, i.e.,
 $$\left[\frac{\partial w}{\partial n}\right](t)=2\dot\Gamma(t).$$
 But $\left[\frac{\partial w}{\partial n}\right]=-[\frac{\partial v}{\partial n}]+[\frac{\partial A^{-1}\varphi}{\partial n}]+[\frac{\partial A^{-1}\sigma}{\partial n}]-[\frac{\partial A^{-1}v}{\partial n}]=-[\frac{\partial v}{\partial n}]$ by Lemma \ref{lemma:invDelta2}, that is the thesis.
\end{proof}

\begin{theorem}
 Assume hypotheses of Theorem \ref{main} with (HP1) and  $s=2$. Then
 $$\left[\frac{\partial Av}{\partial n}\right](t)=-2\dot\Gamma(t)\;\;\;\;\mathcal H^2-\text{a.e. on }\Gamma,$$
 for a.e. $t\in[0,T^*]$.
 \end{theorem}
\begin{proof}
 Denoting again $w:=-v+A^{-2}\varphi+A^{-2}\sigma-A^{-2}v$, equation \eqref{flowlim3} reads $AAw=2\dot\Gamma.$ Thus, applying the same argument of Theorem \ref{main2} to $Aw$ we get the thesis. 
\end{proof}

\section*{Acknowledgements}

The financial support of the FP7-IDEAS-ERC-StG \#256872
(EntroPhase) is gratefully acknowledged by the authors. The present paper 
also benefits from the support of the GNAMPA (Gruppo Nazionale per l'Analisi Matematica, la Probabilit\`a 
e le loro Applicazioni) of INdAM (Istituto Nazionale di Alta Matematica) and the IMATI -- C.N.R. Pavia.

\end{document}